\documentclass{birkjour}
 \newtheorem{thm}{Theorem}[section]
 
 \newtheorem{lem}[thm]{Lemma}
 
 \theoremstyle{definition}
 
 \theoremstyle{remark}
 \newtheorem{rem}[thm]{Remark}
 
 \numberwithin{equation}{section}

\begin{document}

%
%
%
%
%
%
%
%
%

\renewcommand{\div}{{\rm div}\,}
\newcommand{\R}{{\mathbb R}}
\newcommand{\ind}{{\mathbf 1}}
\newcommand{\tu}{\tilde u}
\newcommand{\tuu}{\tilde U}
\newcommand{\tb}{\tilde b}
\newcommand{\bu}{\overline u}
\newcommand{\buu}{\overline U}
\newcommand{\bb}{\overline b}

\title[Optimality of integrability for advection-diffusion]
{Optimality of integrability estimates \\ for advection-diffusion equations}

\author[S.~Bianchini]{Stefano Bianchini}
\address{%
SISSA,
via Bonomea 265,
I-34136 Trieste (Italy)}
\email{bianchin@sissa.it}

\author[M.~Colombo]{Maria Colombo}
\address{%
Institute for Theoretical Studies, ETH Z\"urich,
Clausiusstrasse 47,
CH-8092 Z\"urich (Switzerland)\br
and\br
Institut f\"ur Mathematik, Universit\"at Z\"urich,
Winterthurerstrasse 190,
CH-8057 Z\"urich (Switzerland)}
\email{maria.colombo@eth-its.ethz.ch}

\author[G.~Crippa]{Gianluca Crippa}
\address{%
Departement Mathematik und Informatik, Universit\"at Basel,
Spiegelgasse 1,
CH-4051 Basel (Switzerland)}
\email{gianluca.crippa@unibas.ch}

\author[L.~V.~Spinolo]{Laura V.~Spinolo}
\address{%
IMATI-CNR, 
via Ferrata 1, 
I-27100 Pavia (Italy)}
\email{spinolo@imati.cnr.it}



\date{}
\dedicatory{Dedicated to Alberto Bressan on the occasion of his 60\,$^{th}$ birthday.}

\begin{abstract}
We discuss $L^p$ integrability estimates for the solution~$u$ of the advection-diffusion equation~$\partial_t u + \div(bu) = \Delta u$, where the velocity field~$b \in L^r_t L^q_x$. We first summarize some classical results proving such estimates for certain ranges of the exponents $r$ and $q$. Afterwards we prove the optimality of such ranges by means of new original examples. 
\end{abstract}

\maketitle

\section{Introduction}\label{s:intro}
We consider the advection-diffusion equation
\begin{equation}\label{e:PDE}
\partial_t u + \div(bu) = \Delta u \,,
\end{equation}
where the velocity field $b=b(t,x) : [0,1] \times \R^d \to \R^d$ and the unknown $u = u(t,x) : [0,1] \times \R^d \to \R$. We denote by $u_0 = u_0(x) : \R^d \to \R$ the initial datum for $u$. 

\medskip

When $b \in L^\infty([0,1]\times\R^d)$ and $u_0 \in L^2(\R^d)$ equation~\eqref{e:PDE} possesses a unique solution $u$ in the parabolic class
\begin{equation}\label{e:paraclass}
u \in L^2([0,1];H^1(\R^d)) \,, \quad \text{ with } \quad u' \in L^2([0,1];H^{-1}(\R^d)) \,,
\end{equation}
where we denote by $u'$ the (distributional) derivative of $u$ with respect to time (see for instance~\cite{evans}). To this regularity class for the solution corresponds the following elementary energy estimate. Multiplying~\eqref{e:PDE} times $u$, integrating over $\R^d$, and integrating by parts we obtain
$$
\frac{1}{2} \frac{d}{dt} \int_{\R^d} u^2 \, dx
=
\int_{\R^d} u b \cdot \nabla u \, dx - \int_{\R^d} | \nabla u|^2 \, dx \,.
$$
Since $b$ is bounded we can use Young's inequality to obtain
$$
\frac{1}{2} \frac{d}{dt} \int_{\R^d} u^2 \, dx
\leq 
\| b \|_\infty \left[ \varepsilon \int_{\R^d} |\nabla u|^2 \, dx + C_\varepsilon \int_{\R^d} u^2 \, dx \right] - \int_{\R^d} | \nabla u|^2 \, dx \,,
$$
and by choosing $\varepsilon>0$ sufficiently small we find
$$
\frac{1}{2} \frac{d}{dt} \int_{\R^d} u^2 \, dx
+
\frac{1}{2} \int_{\R^d} | \nabla u|^2 \, dx
\leq 
C \int_{\R^d} u^2 \, dx \,.
$$
In particular, 
\begin{equation}\label{e:gale}
\| u(t,\cdot) \|_{L^2(\R^d)} \leq C \| u_0 \|_{L^2(\R^d)} \,,
\qquad \text{ with $C = C(\| b \|_\infty)$,}
\end{equation}
uniformly for $0 \leq t \leq 1$. In fact, the estimate~\eqref{e:gale} implies at once uniqueness for the solution of~\eqref{e:PDE} in the parabolic class~\eqref{e:paraclass} and, together with the Galerkin approximation method, existence in the same class (see again~\cite{evans} for details). Moreover, if $b$ is smooth, the unique solution $u$ is smooth as well. 

\medskip

On the other hand, a priori estimates in $L^1(\R^d)$ are available without integrability assumptions on $b$. Indeed, in a regular framework, integrating over $\R^d$ the inequality
$$
\partial_t |u| + \div(b|u|) \leq \Delta |u| 
$$
we discover that $\| u(t,\cdot) \|_{L^1(\R^d)} \leq \| u_0 \|_{L^1(\R^d)}$.

\medskip

In this note we discuss a priori integrability estimates for the solution~$u$ of~\eqref{e:PDE} when the velocity field $b$ has some integrability in space and time, but in general is unbounded. For $1 \leq r,q \leq \infty$ we consider a velocity field
$$
b \in L^r([0,1];L^q(\R^d)) \,.
$$
The theory developed in~\cite{lady2} provides estimates for $u$ in $L^2(\R^d)$ uniformly in time provided
\begin{equation}\label{e:range}
\frac{2}{r} + \frac{d}{q} \leq 1\,, \quad
\text{ with } \quad
\left\{ \begin{array}{ll}
\text{$r \in [2,\infty)$ and $q \in (d,\infty]$} & \text{ if $d \geq 2$}\\ \\ 
\text{$r \in [2,4]$ and $q \in [2,\infty]$} & \text{ if $d =1$.}
\end{array}\right.
\end{equation}
When $d\geq 2$ the same holds for $r=\infty$ if we assume $q>d$. We present a proof of such $L^2$ estimate in~\S\ref{s:estimate}, essentially following~\cite{lady2}.

In \S\ref{s:duhamel} we present a different proof of such integrability estimates using Duhamel representation formula, which requires a slightly different range for the exponents than~\eqref{e:range}, see in particular~\eqref{e:coeffi}. 

We remark in passing that~\eqref{e:range} defines the critical integrability for the velocity field $b$ in~\eqref{e:PDE}. In fact, equation~\eqref{e:PDE} is invariant under the scaling
$$
u_\lambda(t,x) = u \big( \lambda t, \sqrt{\lambda}x \big) \,, \qquad
b_\lambda(t,x) = \sqrt{\lambda} \, b \big( \lambda t, \sqrt{\lambda}x \big) \,.
$$
The space $L^r(\R;L^q(\R^d))$ is invariant under the above scaling of $b$ exactly when $2/r + d/q = 1$. 

\medskip

The main contribution of this note is the proof of the {\em sharpness} of the ranges in~\eqref{e:range} and~\eqref{e:coeffi} for integrability estimates on the solution $u$ of~\eqref{e:PDE}. Although several examples are presented in~\cite{lady2} (and in~\cite{lady1} for the related case of elliptic problems), none of them can be easily adapted to the case of equation~\eqref{e:PDE}, as singularities are always present in further coefficients in the equation. We work in the whole space $\R^d$, however the solutions constructed in our examples decay fast at infinity. 

\medskip

In~\S\ref{s:gaussian} we show that if
$$
\frac{2}{r} + \frac{d}{q} > 1 \,, \qquad
\text{ with $1 \leq r,q \leq \infty$} 
$$
then in general $u$ does not enjoy uniformly in time any estimate better than in $L^1(\R^d)$, even for a smooth velocity field and a smooth and compactly supported initial datum. We achieve this using a perturbation argument. More in detail, we fix a (Gaussian) solution of a backward heat equation which converges to a Dirac mass as $t \uparrow 1$. We then ``convert'' the backward heat equation into a forward advection-diffusion equation by defining a suitable velocity field. We finally truncate both the velocity field and the initial datum to gain summability for the velocity field and compact support for the initial datum, and prove that the solution to this perturbed problem still develops a Dirac mass at time~$t=1$.

\medskip

In~\S\ref{s:self1} we address the case $r=\infty$ and $q=d$, which is borderline for the ranges of exponents in~\eqref{e:range} and~\eqref{e:coeffi}. In this case we find solutions by a self-similarity procedure, i.e., we look at solutions depending on the variable $y = x / \sqrt{1-t}$ only. In this context, for $d \geq 1$,  we construct a velocity field in $L^\infty([0,1];L^d(\R^d))$ and a smooth, bounded initial datum in~$L^2(\R^d)$ whose associated solution exits from $L^2(\R^d)$ at time~$t=1$. 

\section{Proof of the $L^2$ estimate under~\eqref{e:range}}\label{s:estimate}

In this section we prove the following theorem:

\begin{thm}\label{t:estimate}
Assume that $b \in L^r([0,1];L^q(\R^d))$, where $r$ and $q$ are as in~\eqref{e:range}. Then every solution $u$ in the parabolic class~\eqref{e:paraclass} of the advection-diffusion equation~\eqref{e:PDE} satisfies for $0 \leq t \leq 1$ the estimate
$$
\| u(t,\cdot) \|_{L^2(\R^d)} \leq C \| u_0 \|_{L^2(\R^d)} \,,
\quad \text{ with $C = C\left(d,r,q,\| b \|_{L^r([0,1];L^q(\R^d))}\right)$.}
$$ 
\end{thm}

\begin{proof}
We start by multiplying the equation~\eqref{e:PDE} times $u$. Integrating over~$\R^d$ and integrating by parts we get
\begin{equation}\label{e:first}
\frac{1}{2} \frac{d}{dt} \int_{\R^d} |u|^2 \, dx + \int_{\R^d} | \nabla u|^2 \, dx \leq
\int_{\R^d} | bu| \, | \nabla u | \, dx \,.
\end{equation}
Given $2 \leq p \leq \infty$ we estimate the right hand side with
\begin{equation}\label{e:conto1}
\begin{aligned}
\int_{\R^d}  | bu| \, | \nabla u | \, dx
& \leq 
\| bu \|_{L^2(\R^d)} \| \nabla u\|_{L^2(\R^d)} \\
& \leq
\| b \|_{L^{\frac{2p}{p-2}}(\R^d)} \| u \|_{L^p(\R^d)} \| \nabla u\|_{L^2(\R^d)} \,.
\end{aligned}
\end{equation}
Recall Gagliardo-Nirenberg-Lady\v zhenskaya's inequality: for $0 \leq \alpha \leq 1$ we can estimate
\begin{equation}\label{e:GNL}
\| u \|_{L^p} \leq C \| \nabla u\|^\alpha_{L^2(\R^d)} \|u\|^{1-\alpha}_{L^2(\R^d)} \,,
\end{equation}
where
\begin{equation}\label{e:GNLcond}
\frac{1}{p} = \frac{1}{2} - \frac{\alpha}{d} \,.
\end{equation}
Notice that when $d=1$ we have the constraint $\alpha \leq 1/2$. Using~\eqref{e:GNL} in~\eqref{e:conto1} we obtain
$$
\int_{\R^d}  | bu| \, | \nabla u | \, dx
\leq
C \| b \|_{L^{\frac{2p}{p-2}}(\R^d)} \|u\|^{1-\alpha}_{L^2(\R^d)} \| \nabla u\|^{1+\alpha}_{L^2(\R^d)} \,.
$$
For $0 \leq \alpha <1$ we can apply Young's inequality to obtain
\begin{equation}\label{e:second}
\int_{\R^d}  | bu| \, | \nabla u | \, dx
\leq
C \left( \| b \|_{L^{\frac{2p}{p-2}}(\R^d)} \|u\|^{1-\alpha}_{L^2(\R^d)} \right)^{\frac{2}{1-\alpha}}
+ \frac{1}{2} \,  \| \nabla u\|^2_{L^2(\R^d)} \,.
\end{equation}
Using~\eqref{e:second} in~\eqref{e:first} we obtain
$$
\frac{1}{2} \frac{d}{dt} \int_{\R^d} u^2 \, dx + \frac{1}{2} \int_{\R^d} | \nabla u|^2 \, dx 
\leq
C \left( \| b \|_{L^{\frac{2p}{p-2}}(\R^d)} \|u\|^{1-\alpha}_{L^2(\R^d)} \right)^{\frac{2}{1-\alpha}} \,,
$$
from which we get in particular
\begin{equation}\label{e:conto2}
\frac{1}{2} \frac{d}{dt} \|u\|^2_{L^2(\R^d)} 
\leq
C \| b \|^{\frac{2}{1-\alpha}}_{L^{\frac{2p}{p-2}}(\R^d)} \|u\|^2_{L^2(\R^d)} \,.
\end{equation}
The estimate in~\eqref{e:conto2} allows use the use of Gronwall's lemma and eventually to establish a uniform in time estimate on~$u$ in $L^2(\R^d)$, provided the following integrability condition is satisfied:
\begin{equation}\label{e:integrability}
\int_0^1 \| b \|^{\frac{2}{1-\alpha}}_{L^{\frac{2p}{p-2}}(\R^d)} \, dt < \infty \,.
\end{equation}
We conclude by showing that, if $r$ and $q$ are as in~\eqref{e:range}, then~\eqref{e:integrability} holds. Denoting by
$$
r = \frac{2}{1-\alpha} 
\qquad \text{ and } \qquad
q = \frac{2p}{p-2}\,,
$$
we observe that for $d\geq2$ we have $r \in [2,\infty)$, since $0 \leq \alpha < 1$, while for~$d=1$ the constraint $0\leq \alpha\leq 1/2$ implies $r \in [2,4]$. 
We compute
$$
\frac{2}{r} + \frac{d}{q} 
= \frac{2}{\frac{2}{1-\alpha}} + \frac{d}{\frac{2p}{p-2}}
= 1 - \alpha + d \left( \frac{1}{2} - \frac{1}{p} \right)
= 1 \,,
$$
where we have used~\eqref{e:GNLcond} in the last equality. In particular we also find the range for $q$ in~\eqref{e:range}. Since we are dealing with a bounded interval of times, we can always increase the assumption on the integrability in time. This proves that, if $r$ and $q$ are as in~\eqref{e:range}, then~\eqref{e:integrability} holds, as wanted.
\end{proof} 

\begin{rem}
We remark that with analogue computations it is possible to prove higher integrability estimates for the solution, still under the same assumption as in Theorem~\ref{t:estimate} that the velocity field belongs to $L^r([0,1];L^q(\R^d))$ with $r$ and $q$ as in~\eqref{e:range}. In fact, one can easily check that for any $\gamma >1$ the estimate~\eqref{e:conto2} can be improved to
$$
\frac{1}{2} \frac{d}{dt} \||u|^\gamma\|^2_{L^2(\R^d)} 
\leq
C \| b \|^{\frac{2}{1-\alpha}}_{L^{\frac{2p}{p-2}}(\R^d)} \||u|^\gamma\|^2_{L^2(\R^d)} \,,
$$
and therefore 
$$
\| u(t,\cdot) \|_{L^{2\gamma}(\R^d)} \leq C \| u_0 \|_{L^{2\gamma}(\R^d)} \,,
\quad \text{ with $C = C\left(\gamma,d,r,q,\| b \|_{L^r([0,1];L^q(\R^d))}\right)$.}
$$ 

\end{rem}

\section{An alternative proof using Duhamel formula}
\label{s:duhamel}

We now provide a different proof of the estimate in Theorem~\ref{t:estimate}, which allows us to get a slightly different range for the exponents $r$ and $q$ than in~\eqref{e:range}. In particular, in dimension $d=1$ Theorem~\ref{l:swap} below states that, 
if the initial datum $u_0$ is bounded, than we can get the $L^2$ estimate for~$u$ if~$b \in L^r ([0, 1]; L^q (\R^d))$ provided that 
$$
    \frac{2}{r} + \frac{1}{q} <1 \,. 
$$
Note that compared to assumption~\eqref{e:range} we have lost the limit case $2/r + 1/q =1$, but we have 
removed the restriction that $q \ge 2$.
\begin{thm}
\label{l:swap}
Assume that $u_0 \in L^1 (\R^d) \cap L^\infty (\R^d)$, that $b \in L^r ([0, 1]; L^q (\R^d))$ and that 
\begin{equation}
\label{e:coeffi}
      \frac{2}{r} + \frac{d}{q} <1 \,, \qquad
      \text{ with $r \in [2,\infty)$ and $q \in (d,\infty]$.} 
\end{equation}
Then every solution $u$ in the parabolic class~\eqref{e:paraclass} of the advection-diffusion equation~\eqref{e:PDE} satisfies
for every $p \in [1, + \infty[$ and $0 \leq t \leq 1$ the estimate 
\begin{equation}
\label{e:interpolo}
   \| u(t, \cdot) \|_{L^p (\R^d)} \leq 
   C \|u_0 \|_{L^\infty (\R^d)}^{\frac{p-1}{p}} \| u_0 \|_{L^1 (\R^d)}^{\frac{1}{p}} \,, 
   \quad C = C(p, r, q, d,\| b\|_{ L^r ([0, 1]; L^q (\R^d))})
\end{equation}
and furthermore for $0 \leq t \leq 1$ we have 
\begin{equation}
\label{e:interpolo2}
   \| u(t, \cdot) \|_{L^\infty (\R^d)} \leq 
   C \|u_0 \|_{L^\infty (\R^d)} \,, 
   \quad C = C( r, q, d,\| b\|_{ L^r ([0, 1]; L^q (\R^d))}) \,. 
\end{equation}
\end{thm}

Observe that the estimates in~\eqref{e:interpolo} and~\eqref{e:interpolo2} can be seen as a priori estimates on solutions of~\eqref{e:PDE}. 

The proof of Theorem~\ref{l:swap} is based on the following result. It allows to estimate the $L^p$-norm of a solution of~\eqref{e:PDE} in terms of the same norm of the initial datum provided that $p$ is large enough with respect to $q$, as required in the second inequality in~\eqref{e:d:rel} below.
Notice that this second condition is natural in this context, since it allows to give a distributional meaning to the second term in~\eqref{e:PDE}. 

\begin{lem}
\label{l:ellepi}
Assume $b \in L^r ([0, 1]; L^q (\R^d))$ and that the initial datum satisfies ${u_0 \in L^p (\R^d)}$.
Assume furthermore that $p \in [1, + \infty]$ and that 
\begin{equation}
\label{e:d:rel}
   \frac{2}{r} + \frac{d}{q} <1 \qquad \text{ and } \qquad
    \frac{1}{p} + \frac{1}{q} \leq 1\,. 
\end{equation}
Then for $0 \leq t \leq 1 $ we have 
\begin{equation}
\label{e:d:ellepi}
  \| u(t, \cdot) \|_{L^p (\R^d)} \leq C \| u_0 \|_{L^p(\R^d)}\,, 
  \quad 
  C = C(r, q, d,\| b\|_{ L^r ([0, 1]; L^q (\R^d))}) \,.  
\end{equation}
\end{lem}

In the following proof we denote by $C_{a, b}$ a constant depending on the quantities $a$ and $b$ only. Its precise value can vary from occurrence to occurrence. Also, we only provide a formal proof, which can be made rigorous by relying on suitable approximation arguments. 

\begin{proof}[Proof of Lemma~\ref{l:ellepi}]

First, we define the value $\tau$ by setting 
\begin{equation}
\label{e:d:tau}
    \tau : = \inf \big\{ t \in [0, 1]: \; \| u(t, \cdot) \|_{L^p(\R^d)} > 2 \| u_0 \|_{L^p(\R^d)} \big\} \,. 
\end{equation}
If $\tau = + \infty$, then~\eqref{e:d:ellepi} holds, so we can assume $\tau \leq 1$, which implies 
\begin{equation}
\label{e:d:contra}
   \| u(\tau, \cdot) \|_{L^p(\R^d)} = 2 \| u_0 \|_{L^p(\R^d)} \,. 
\end{equation}
Next, we point out that the Duhamel representation formula implies that 
\begin{equation}
\label{e:d:duhamel}
u(\tau, \cdot) = G(\tau, \cdot) \ast u_0 - \int_0^{\tau}
\nabla G(\tau-s, \cdot) \ast \big[ bu \big] (s, \cdot) ds \,.  
\end{equation}
In the previous expression, $G(t,x) := \frac{1}{(4 \pi t)^{d/2}} \, e^{-|x|^2 / (4 t)}$
is the heat kernel and~$\ast$ denotes the convolution computed with respect to the space variable only. 
By using the Young Theorem on convolutions we get 
\begin{equation}
\label{e:d:duhamel2}
\begin{aligned}
\| u(\tau, \cdot) \|_{L^p(\R^d)} & \leq \| G(\tau, \cdot) \ast u_0 \| _{L^p(\R^d)} + \left\| \int_0^{\tau}
\nabla G(\tau-s, \cdot) \ast \big[ bu \big] (s, \cdot) ds \right\|_{L^p(\R^d)}  \\ 
& \leq \| G(\tau, \cdot) \ast u_0 \| _{L^p(\R^d)} +  \int_0^{\tau} \| 
\nabla G(\tau-s, \cdot) \ast \big[ bu \big] (s, \cdot) \|_{L^p(\R^d)}  ds \\
& \stackrel{\text{Young}}{\leq} \| u_0 \| _{L^p(\R^d)} +
\int_0^\tau \| 
\nabla G(\tau-s, \cdot) \|_{L^{q^\ast}(\R^d)} \| b u (s, \cdot) \|_{L^b(\R^d)} ds 
\end{aligned}
\end{equation}
provided that $b$, $q^*\geq 1$ are chosen such that
\begin{equation}
\label{e:d:b}
\frac{1}{p} = \frac{1}{q^\ast} + \frac{1}{b} -1 \,.
\end{equation}
Next, we use the H\"older inequality 
\begin{equation}
\label{e:d:holder}
  \| b u (s, \cdot) \|_{L^b(\R^d)} \leq \| b\|_{L^q(\R^d)} \|u \|_{L^p(\R^d)} \,, 
  \quad \text{if \; \; $\frac{1}{b} = \frac{1}{q} + \frac{1}{p}$}
\end{equation}
and we recall that, owing to~\eqref{e:d:tau}, 
$$
 \| u (s, \cdot)\|_{L^p(\R^d)} \leq 2 \| u_0\|_{L^p(\R^d)}
 \quad \text{for every $s \leq \tau$.}
 $$  
We plug the above inequalities into~\eqref{e:d:duhamel2} and we arrive at 
\begin{equation}
\label{e:d:duhamel3}
\begin{aligned}
& \| u(\tau, \cdot) \|_{L^p(\R^d)}  \leq 
\| u_0 \| _{L^p(\R^d)} \\
& \qquad \qquad \qquad \qquad + 2 \|u_0\|_{L^p(\R^d)} \! \! \! 
\int_0^\tau \! \! \| 
\nabla G(\tau-s, \cdot) \|_{L^{q^\ast}(\R^d )} \! \| b  (s, \cdot) \|_{L^q(\R^d)} ds\\
& \stackrel{\text{H\"older}}{\leq} 
\| u_0 \| _{L^p(\R^d)} + 2 \|u_0\|_{L^p(\R^d)} \| 
\nabla G \|_{L^{r^\ast}([0, \tau]; L^{q^\ast}(\R^d))} 
\| b  \|_{L^{r}([0, 1]; L^q(\R^d))} \,,
\end{aligned}
\end{equation}
provided that 
\begin{equation}
\label{e:d:erre}
\frac{1}{r} + \frac{1}{r^\ast} =1 \,. 
\end{equation}
We now focus on the term involving $\nabla G$: since 
$$
    | \nabla G (t, x) | = C_d \frac{|x|}{t^{d/2+1}} e^{-\frac{-|x|^2}{4t}} \,,
$$
then by using spherical coordinates we get 
\begin{equation*}
\begin{aligned}
    \| \nabla G (t, \cdot) \|_{L^{q^\ast} (\R^d)} & = 
   C_{d, q^\ast} \left( \int_0^\infty 
   \rho^{d-1} \frac{\rho^{q^\ast}}{t^{q^\ast (d/2+1)}} e^{-\frac{-q^\ast \rho^2}{4t}} d\rho 
   \right)^{1/q^\ast} \\
   & \stackrel{z = \rho /2 t^{1/2}}{=} C_{d, q^\ast}
   \left( \int_0^\infty z^{d-1+q^\ast} e^{- q^\ast z^2}
  t^{\frac{d- q^{\ast} (d+1)}{2}}  d z
   \right)^{1/q^\ast} \\ & =
   C_{d, q^\ast} t^{\frac{d- q^{\ast} (d+1)}{2q^\ast}} \,. 
\end{aligned}
\end{equation*}
This implies that 
\begin{equation}
\label{e:d:nablaG}
\|
\nabla G \|_{L^{r^\ast}([0, \tau]; L^{q^\ast}(\R^d))} = 
C_{d, q^\ast, r^\ast} \tau^{\frac{\alpha +1}{r^\ast}  }
\end{equation}
provided that 
\begin{equation}
\label{e:d:alpha}
   \alpha = \big[d- q^{\ast} (d+1)\big]   \frac{r^\ast}{2q^\ast} > - 1 \,.
\end{equation}
We now plug the inequality~\eqref{e:d:nablaG} into~\eqref{e:d:duhamel3} and we use~\eqref{e:d:contra}. We conclude that 
$$
   \| u(\tau, \cdot) \|_{L^p(\R^d)}  = 2 \| u_0\|_{L^p(\R^d)}
   \leq \| u_0\|_{L^p(\R^d)} \Big[ 1 + 
   \| b  \|_{L^{r}([0, 1]; L^q(\R^d))} 
   C_{d, q^\ast, r^\ast} \tau^{\frac{\alpha +1}{r^\ast}  }
   \Big]\,,
$$
which implies that 
$$
    \tau \ge \beta:=  \left(   \frac{1}{C_{d, q^\ast, r^\ast}  
    \| b  \|_{L^{r}([0, 1]; L^q(\R^d))} }\right)^{\frac{r^\ast}{\alpha +1 }} \,.
$$
This implies that~\eqref{e:d:ellepi} holds on $[0, \beta]$ with $C=2$. 
Note furthermore that $\beta$ does \emph{not} depend on $\| u_0\|_{L^p(\R^d)}$ and hence we can iterate the above argument. More precisely, let $n$ be the integer part of $1/\beta$, then we apply the above argument on the intervals $[0, \beta],$ $[\beta, 2 \beta]$, $[2 \beta, 3 \beta], \dots [n \beta, 1]$
and get that~\eqref{e:d:ellepi} holds on $[0, 1]$ with $C=2^{n+1}$. 

To conclude, we discuss the range of $r$ and $q$. By combining~\eqref{e:d:b} and~\eqref{e:d:holder} we get that 
\begin{equation}
\label{e:d:cu}
   \frac{1}{q} + \frac{1}{q^\ast} =1
\end{equation}
and by plugging the above relation and~\eqref{e:d:erre} into~\eqref{e:d:alpha} we arrive at 
$$
   \frac{2}{r} + \frac{d}{q} <1 \,.
$$
Finally, we require that the number $b$ in~\eqref{e:d:b} satisfies $b \in [1, + \infty]$ and owing to~\eqref{e:d:cu} it suffices to impose that 
$$
   \frac{1}{b} = \frac{1}{p} + \frac{1}{q} \leq 1\,. \eqno\qedhere
$$
\end{proof}
We can now give the 
\begin{proof}[Proof of Theorem~\ref{l:swap}]
We rely on an elementary interpolation argument. First, we recall that 
\begin{equation}
\label{e:elleuno}
  \| u(t, \cdot) \|_{L^1 (\R^d)} \leq 
  \| u_0 \|_{L^1 (\R^d)}, \quad \text{ for $0 \leq t \leq 1$.} 
\end{equation}
Next, we apply~\eqref{e:d:ellepi} with $p=+ \infty$ and we obtain~\eqref{e:interpolo2}. Note that since $p=+ \infty$  the second condition in~\eqref{e:d:rel} is satisfied in this case for every ${q \in [1, + \infty]}$. By interpolating between~\eqref{e:interpolo2} and~\eqref{e:elleuno} we then arrive at~\eqref{e:interpolo}.
\end{proof}

\section{First example: perturbation of a Gaussian}\label{s:gaussian}

In this section we prove the following result:

\begin{thm}
\label{t:pert}
Let $1 \leq r,q \leq \infty$ be such that
\begin{equation}\label{e:negrange}
\frac{d}{q} + \frac{2}{r} > 1 \,.
\end{equation}
Then there exists a velocity field $b \in L^r ([0,1);L^q(\R^d)) \cap C^
\infty([0,1) \times \R^d)$, an 
initial datum $u_0 \in C^\infty_c(\R^d)$, and a solution $u \in C^\infty([0,1) \times \R^d)$
of the Cauchy problem for~\eqref{e:PDE} which has no uniform in time integrability bounds better
than in~$L^1(\R^d)$, or equivalently, $u$ is not uniformly in time equi-integrable on~$\R^d$.
In particular, for any $1 < p \leq \infty$ we have
$$
\lim_{t \to 1^-} \| u(t,\cdot) \|_{L^p(\R^d)} = \infty \,.
$$
\end{thm}

\begin{rem}
The velocity field $b$ constructed in the proof in fact belongs to~$C^\infty_c([0,1-\tau] \times \R^d)$ for any $0<\tau<1$. In particular, the solution $u$ is the unique one in the ``local parabolic class''
$$
u \in L^2_{\rm loc}([0,1);H^1(\R^d)) \,, \quad \text{ with } \quad u' \in L^2_{\rm loc}([0,1);H^{-1}(\R^d)) \,,
$$
\end{rem}

\begin{proof}[Proof of Theorem~\ref{t:pert}] 
Rather than ``giving the formulas'' for the velocity field and the solution we proceed
step by step with their construction. The main idea is to obtain $u$ as a small $L^1$
perturbation of a solution of a backward heat equation which develops a Dirac mass at
$x=0$ for $t=1$. 

\medskip

\noindent {\bf Step 1: Backward heat equation.} 
Let
$$
G(t,x) := \frac{1}{[4 \pi(1-t)]^{d/2}} \, e^{-\frac{|x|^2}{4(1-t)}} \,, 
\qquad 0 \leq t < 1\,, \, x\in\R^d\,.
$$
It is immediate to check that
\begin{enumerate}
\item $\int_{\R^d} G(t,x) \, dx =1$ for any $0 \leq t < 1$ and $G(t,\cdot) \rightharpoonup \delta_0$ as $t \uparrow 1$;
\item $G$ solves the Cauchy problem
\begin{equation}\label{e:backPDE}
\begin{cases}
\partial_t G = - \Delta G \\
G(0,x) = \displaystyle \frac{1}{(4 \pi)^{d/2}} \, e^{-\frac{|x|^2}{4}} \in L^1 \cap L^\infty \cap C^\infty (\R^d) \,.
\end{cases}
\end{equation}
\end{enumerate} 

\medskip

\noindent {\bf Step 2: Definition of the velocity field.}
Ideally, we would like to transform the backward heat equation in~\eqref{e:backPDE} into the advection-diffusion equation
\begin{equation}\label{e:BG}
\partial_t G + \div (BG) = \Delta G \,,
\end{equation}
for some $B=B(t,x)$. In order to achieve this exactly, we would need to have
$$
\div (BG) = 2 \Delta G \,,
$$
which would require
\begin{equation}\label{e:B}
B(t,x) = - \frac{x}{1-t} \,.
\end{equation}
Indeed, with $B$ defined as in~\eqref{e:B} we have
\begin{equation}\label{e:gradB}
BG = - \frac{x}{1-t} \, \frac{1}{[4 \pi(1-t)]^{d/2}} \, e^{-\frac{|x|^2}{4(1-t)}} = 2 \nabla G \,.
\end{equation}
However, the problem is that with this definition $B$ grows linearly in $x$ and therefore has no 
integrability on~$\R^d$. For this reason we need to truncate it. 

Let $\varphi : [0,\infty) \to [0,\infty)$ be a nonnegative decreasing function with Lipschitz constant less or equal 
than $2$ and such that
\begin{equation}\label{e:varphi}
\varphi(y) = \left\{ \begin{array}{ll}
1 & \text{ for $0 \leq y \leq 1$} \\
0 & \text{ for $y \geq 2$.}
\end{array}\right.
\end{equation}
Given two parameters $\gamma$, $\beta>0$
to be suitably chosen in the following, we let
$$
\chi(t,x) = \varphi \left(\frac{|x|}{\gamma(1-t)^\beta} \right)
$$
and 
$$
b(t,x) = B(t,x) \chi(t,x) \,. 
$$
Moreover we also want to make the initial datum compactly supported. To this aim we choose $\phi\in C^\infty_c(\R^d)$ with $\phi \equiv 1$ on a large ball so that 
\begin{equation}\label{e:smallinitial}
\| G(0,\cdot) - G(0,\cdot)\phi\|_{L^1(\R^d)} < 1/2 \,.
\end{equation} 

\medskip

\noindent {\bf Step 3: Strategy for the next steps.}
The function $G$ is a solution of the transport-diffusion equation~\eqref{e:BG} with 
velocity field $B$ and with initial datum $G(0,\cdot)$, but not with their truncations $b$ and $G(0,\cdot)\phi$. However, we can write the 
solution~$u$ of the Cauchy problem 
\begin{equation}\label{e:PDEbu}
\begin{cases}
\partial_t u + \div(bu) = \Delta u \\
u(0,x) = G(0,x) \phi(x) = \displaystyle \frac{1}{(4 \pi)^{d/2}} \, e^{-\frac{|x|^2}{4}} \phi(x) \in C_c^\infty (\R^d) 
\end{cases}
\end{equation}
as 
$$
u = G + v \,, \qquad \text{$v$ ``perturbation''.}
$$
Our goal is then to choose the parameters $\gamma$ and $\beta$ in such a way that
\begin{enumerate}
\item $b \in L^r ([0,1);L^q(\R^d))$, for $r$ and $q$ as in~\eqref{e:negrange}, and
\item the norm of $v$ in $L^1([0,1);L^1(\R^d))$ is strictly less than $1$, so that $v$ cannot cancel out the concentration of $g$ for $t \uparrow 1$. 
\end{enumerate} 

\medskip

\noindent {\bf Step 4: Smallness of $v$.} 
The perturbation $v$ solves the Cauchy problem
$$
\begin{cases}
\partial_t v + \div (bv) - \Delta v = 
- \big( \partial_t G + \div (bG) - \Delta G \big) \\
v(0,x) = G(0,x) (\phi(x)-1) \,.
\end{cases}
$$
Observing that
$$
\begin{aligned}
\| v \|&_{L^\infty([0,1];L^1(\R^d))} \\
& \leq
\| G(0,x) (\phi(x)-1) \|_{L^1(\R^d)} + \int_0^1 \| \partial_t G + \div (bG) - \Delta G \|_{L^1(\R^d)} \, dt
\end{aligned}
$$
and recalling~\eqref{e:smallinitial}, we see that we need to choose $\gamma$ and $\beta$ so that
\begin{equation}\label{e:smallsource}
\big\| \partial_t G + \div (bG) - \Delta G \big\|_{L^1([0,1);L^1(\R^d))} < 1/2 \,.
\end{equation}
We proceed in several sub-steps.

\noindent {\it A preliminary computation.} We compute first of all
\begin{equation}\label{e:divergence}
\begin{aligned}
-\big( \partial_t G + \div (bG) - \Delta G \big)
& =
-\big( \partial_t G + \Delta G + \div (bG - 2 \nabla G) \big) \\
& = 
- \div (bG - 2 \nabla G) \\
& =
- \div ( BG - (1-\chi)BG -2\nabla G) \\
& =
\div (B \tilde \chi G )\,,
\end{aligned}
\end{equation}
where we have set 
\begin{equation}\label{e:chitilde}
\tilde \chi(t,x) = 1-\chi(t,x) 
\end{equation}
and we have used~\eqref{e:gradB}. 

\noindent {\it Positivity of the quantity $\div(BG)\tilde\chi$.} Using again~\eqref{e:gradB} we see
that $\div (BG) = 2\Delta G$. By direct inspection of the expression of $\Delta G$
we obtain that
$$
\Delta G(t,x) \geq 0 \quad \Longleftrightarrow \quad
\frac{|x|^2}{1-t}  \geq 2d \,.
$$
Recalling~\eqref{e:varphi} and~\eqref{e:chitilde} 
we discover that $\div(BG)\tilde \chi \geq 0$ is therefore implied by
\begin{equation}\label{e:prebeta}
\sqrt{2d} \sqrt{1-t} \leq \gamma (1-t)^\beta 
\qquad 
\text{ for any $0 \leq t \leq 1$.}
\end{equation}
Condition~\eqref{e:prebeta} is guaranteed if
\begin{equation}\label{e:condbeta}
0 < \beta < \frac{1}{2}  \qquad \text{ and } \qquad \gamma \geq \sqrt{2d}\,,
\end{equation}
condition that we assume from now on in our argument. 

\noindent {\it Computation of the norm of the source.} Using the computation in~\eqref{e:divergence} 
we can estimate the norm of the source as
\begin{equation}\label{e:newdiv}
\begin{aligned}
\big\| \partial_t G & + \div (bG) - \Delta G \big\|_{L^1([0,1);L^1(\R^d))} \\
& \stackrel{\eqref{e:divergence}}{=}
\int_0^1 \int_{\R^d} | \div (B \tilde \chi G) | \, dx \, dt \\
& \leq
\int_0^1 \int_{\R^d} | \div (B G) \tilde \chi | \, dx \, dt
+ \int_0^1 \int_{\R^d} | BG | | \nabla \tilde \chi | \, dx \, dt  \\
& \stackrel{(\star)}{=}
\int_0^1 \int_{\R^d} \div (B G) \tilde \chi  \, dx \, dt
+ \int_0^1 \int_{\R^d} | BG | | \nabla \tilde \chi | \, dx \, dt  \\
& \leq
2 \int_0^1 \int_{\R^d} | BG | | \nabla \tilde \chi | \, dx \, dt \,, 
\end{aligned}
\end{equation}
where in $(\star)$ we have used the positivity of $\div(BG)\tilde\chi$ 
and in the last inequality we have used the decay of $G$ at infinity. 
Observing that
$$
|\nabla \tilde \chi (t,x) |  \leq \frac{2}{\gamma(1-t)^\beta} \ind_{\{(x,t): \gamma(1-t)^\beta\leq {|x|}\leq 2{\gamma(1-t)^\beta}\}}
$$
we can estimate the right hand side in~\eqref{e:newdiv}
by employing spherical coordinates and by using the monotonicity 
of the  functions involved as
\begin{equation}\label{e:gammawill}
\begin{aligned}
\big\| \partial_t G & + \div (bG) - \Delta G \big\|_{L^1([0,1);L^1(\R^d))} \\
&\leq 
C_d \int_0^1 \int_{\gamma(1-t)^\beta}^{2\gamma(1-t)^\beta} 
\frac{r}{1-t} \cdot \frac{e^{-\frac{r^2}{4(1-t)}}}{\left[4 \pi (1-t)\right]^{d/2}} \cdot  
\frac{2}{\gamma(1-t)^{\beta}}\, r^{d-1} \, dr \,dt
\\
&\leq 
C_d \int_0^1\frac{(2\gamma(1-t)^\beta)^d}{1-t} \cdot 
\frac{e^{-\frac{\gamma^2(1-t)^{2\beta}}{4(1-t)}}}{\left[4 \pi (1-t)\right]^{d/2}}   \, dt 
\\
&= C_d \gamma^d\int_0^1\frac{1}{(1-t)^{1+\frac d 2 - d \beta}} \cdot e^{-\frac{\gamma^2(1-t)^{2\beta}}{4(1-t)}} \, dt .
\end{aligned}
\end{equation}
Recalling \eqref{e:condbeta} we can use the change of variable
$$
s = \frac{\gamma^2}{4} (1-t)^{2\beta-1}
$$
and estimate the right hand side of~\eqref{e:gammawill} with
$$
\begin{aligned}
\big\| \partial_t G + \div (bG) - \Delta G \big\|_{L^1([0,1);L^1(\R^d))}
& \leq 
C_{d,\beta} \gamma^{\alpha} \int_{\gamma^2/4}^{\infty} s^{\tilde\alpha} e^{-s} \, ds \\
& \stackrel{(\star)}{\leq}
C_{d,\beta} \gamma^{\alpha} \int_{\gamma^2/4}^{\infty} e^{s/2} e^{-s} \, ds \\
& = C_{d,\beta} \gamma^{\alpha} e^{-\gamma^2/8} \,,
\end{aligned}
$$
where $\alpha=\alpha(d,\beta)$ and $\tilde\alpha=\tilde\alpha(d,\beta)$ are real numbers 
and the estimate $(\star)$ holds for $\gamma$ sufficiently large. We can therefore obtain~\eqref{e:smallsource} by choosing $\gamma$ sufficiently large.

\medskip

\noindent {\bf Step 5: Integrability of $b$.} We can compute
$$
\| b(t,\cdot) \|^q_{L^q(\R^d)} 
= \int_{\{ |x| \leq \gamma (1-t)^\beta \}} \frac{|x|^q}{(1-t)^q} \, dx
= C_{d,\gamma} (1-t)^{\beta d + q(\beta-1)} \,,
$$
therefore
$$
\| b \|^r_{L^r ([0,1);L^q(\R^d))} = \int_0^1 (1-t)^{ \left(\frac{\beta d}{q} + \beta -1\right) r} \, dt \,.
$$
Hence $b \in L^r ([0,1);L^q(\R^d))$ if $\left(\frac{\beta d}{q} + \beta -1\right) r > -1$, which can be rewritten as
$$
\left( \frac{\beta}{1-\beta} \right) \frac{d}{q} + \left( \frac{1/2}{1-\beta} \right) \frac{2}{r} > 1 \,.
$$
Observing that as $\beta \uparrow 1/2$ one has
$$
\frac{\beta}{1-\beta} \to 1 
\qquad \text{ and } \qquad
\frac{1/2}{1-\beta} \to 1 
$$
we conclude that we can choose $0<\beta<1/2$ so that $b \in L^r ([0,1);L^q(\R^d))$ as long as $r$ and $q$ are 
as in~\eqref{e:negrange}. 
\end{proof} 

\section{Second example: self-similar solutions}\label{s:self1}

In this section we look at the case $r=\infty$ and $q=d$, which is borderline for the admissible ranges of exponents in~\eqref{e:range} and~\eqref{e:coeffi}. For any $d \geq 1$ we construct a velocity field in $L^\infty([0,1];L^d(\R^d))$ and a smooth, bounded initial datum in~$L^2(\R^d)$ whose associated solution exits from $L^2(\R^d)$ at the final time~$t=1$. In fact, one can easily check that the velocity fields constructed in this section also belong to all spaces described by~\eqref{e:negrange}. We employ here a self-similar construction, i.e., we look at solutions depending on the variable $y = x / \sqrt{1-t}$ only. This procedure has been inspired by the results in~\cite{mooney}.

\subsection{The one dimensional case.} 

We start by discussing the one dimensional case.

\begin{thm}\label{thm:counterex-self-similar}
There exists a velocity field $b \in L^\infty((0,1); L^1(\R)) \cap C^\infty([0,1)\times \R)$  and a smooth solution 
$u \in C^\infty([0,1)\times \R) \cap L^2((0,1); L^2(\R))$ of~\eqref{e:PDE}
with $u_0 \in (L^2 \cap L^\infty \cap C^\infty)(\R)$ such that 
$$
\limsup_{t\to 1^-} u(t,0) = \infty \qquad \text{ and } \qquad \lim_{t\to 1^-} \|u(t,\cdot) \|_{L^2(\R)} = \infty \,.
$$
\end{thm}

\begin{proof} For simplicity, we translate the time variable and we build a solution on~$(-1,0)$ which is bounded and smooth at $t=-1$ and with a blow-up at time $t=0$. 

\medskip

\noindent {\bf Step 1: Self-similar solutions.}
Due to the scaling properties of the equation, we look for a solution of \eqref{eqn:our} of the form
\begin{equation}
\label{eqn:self-similar}
u = \frac{1}{(-t)^\alpha} \tu\Big(\frac{x}{(-t)^{1/2}}\Big) 
\qquad \text{ and } \qquad
b = \frac{1}{(-t)^{1/2}} \tb\Big(\frac{x}{(-t)^{1/2}}\Big)\,,
\end{equation}
for some $\alpha>0$ to be chosen later.
Noticing that
$$
\begin{aligned}
\partial_t u 
& =
 \frac{\alpha}{(-t)^{\alpha+1}} \tu\Big(\frac{x}{(-t)^{1/2}}\Big)
+\frac{1}{2(-t)^{\alpha+1}} \cdot \frac{x}{(-t)^{1/2}} \tu'\Big(\frac{x}{(-t)^{1/2}}\Big)\,, \\
\partial_x(bu) 
&= 
\frac{1}{(-t)^{\alpha+1}} (\tb\tu)'\Big(\frac{x}{(-t)^{1/2}}\Big)\,, \\
\partial_{xx} u 
&= 
\frac{1}{(-t)^{\alpha+1}} \tu''\Big(\frac{x}{(-t)^{1/2}}\Big) \,,
\end{aligned}
$$
we conclude that the couple $(u,b)$ solves \eqref{e:PDE}, namely 
\begin{equation}
\label{eqn:our}
\partial_t u+ \partial_x(bu) = \partial_{xx} u\,,
\end{equation}
if and only if the couple $(\tu, \tb)$ solves for any $y\in \R$
\begin{equation}\label{e:tilde}
\alpha \tu(y) + \frac 1 2 y \tu'(y)+ (\tb\tu)'(y)= \tu''(y) \,.
\end{equation}
Let $\tuu$ be any primitive of $\tu$. Integrating once,~\eqref{e:tilde} is satisfied if and only if
$$
\Big(\alpha- \frac 1 2 \Big) \tuu(y) + \frac 1 2 y \tu(y)+ (\tb\tu)(y)= \tu'(y)\,.
$$

\medskip

\noindent {\bf Step 2: Construction of $\tu$ and $\tb$.}
We claim that we can find a smooth function~$\tuu$, such that $\tuu'= \tu$ is bounded, strictly positive on $\R$ and belongs to~$L^2(\R)$, in such a way that setting
\begin{equation}
\label{eqn:find-b}
\tb = \frac{\tu'- \Big(\alpha- \frac 1 2 \Big) \tuu(y) - \frac 1 2 y \tu}{\tu}
\end{equation}
it holds $\tb \in L^1(\R)$.

We first notice that finding such a function $\tuu$ on a given compact interval is not a problem. 
Actually, fixing any smooth positive $\tuu$ with  $\tuu'= \tu$ strictly positive and defining $\tb$ 
consequently would do the job. Therefore, we focus on finding the strictly increasing function 
$\tuu$ for $|y|>M$, for some~$M>1$ to be fixed later. Once this is done, we define $\tuu $ in 
$[-M, M]$ as a smooth, increasing extension of the previous function with $\tuu'= \tu$ strictly positive. 
Since we will look for an even function $\tu$, we restrict ourselves to finding $\tuu$ on $[M,\infty)$. 
We look for a solution of the form
$$
\tuu (y) = y^{\gamma} + Cy^{\gamma-2} \qquad \mbox{for } y >M
$$
for some $\gamma \in (0,1)$, $C \geq 0$, $M>1$ to be chosen later.

With this choice, for $y>M$ we have 
$$
\tu (y) = \gamma y^{\gamma-1} + C (\gamma -2) y^{\gamma-3}
$$
and
$$
\tb( y) = \frac{ \big( \gamma (\gamma-1) - C\alpha+\frac {3C} 2 - \frac C 2 \gamma \big)y^{-1} + C (\gamma-2) (\gamma-3) y^{-3}  + \big(-\alpha+\frac 1 2 -\frac \gamma 2 \big) y
}{\gamma+ C(\gamma-2) y^{-2}}\,.
$$

Next, we choose the coefficients in order for $\tb$ to have the best possible decay at infinity. In particular we require
$$ 
\gamma (\gamma-1) - C\alpha+\frac {3C} 2 - \frac C 2 \gamma= 0
\qquad \text{ and } \qquad
 -\alpha+\frac 1 2 -\frac \gamma 2= 0 \,,
$$
namely, we choose $\alpha$ and $C$ in terms of $\gamma$ as
$$
\alpha = \frac{1-\gamma}2
\qquad \text{ and } \qquad 
C = \gamma(1-\gamma)\,.
$$
Finally, we choose $\gamma= 1/4$ (remarking that this choice is quite arbitrary) and $M=2$, so that 
the final choice of the parameters is
$$
\gamma =\frac14\,, \qquad \alpha= \frac 38, \qquad \text{ and } \qquad C= \frac {3}{16} \,.
$$
With this choice for $y>2$ we have 
$$
\begin{aligned}
\tuu (y) 
&= 
y^{1/4} + \frac 3 {16}y^{-7/4}\,, \\
\tu (y) 
&= 
\frac 1 4 y^{-3/4} -\frac{21}{64} y^{-11/4}\,, \\
\tb( y) 
&= 
\frac{231 y^{-3} }{4(16-21y^{-2})}\,.
\end{aligned}
$$

Finally, we consider a smooth function $\tuu$ such that
$$
\tuu(y) = \left\{ \begin{array}{ll}
y^{1/4} + \frac 3 {16}y^{-7/4} & \text{for $y>2$} \\ \\ 
- (-y)^{1/4} - \frac 3 {16}(-y)^{-7/4} & \text{for $y<-2$,}
\end{array}\right.
$$
and is smooth, with strictly positive first derivative in $\R$. We notice that $\tu = \tuu' \in L^2(\R)$. Defining eventually $\tb$ as in \eqref{eqn:find-b}, we notice that it decays as~$y^{-3}$ at $\infty$, so it belongs to $L^1(\R)$.

\medskip

\noindent {\bf Step 3: Bounds on $b$ and finite-time blow-up of $u$.} 
From the properties of $\tu$ and $\tb$ in Step~2, it is clear that the functions $b$ and $u$ defined in \eqref{eqn:self-similar} belong to $C^\infty([0,1)\times \R)$. Moreover, the function $u$ blows up for $t=0$:
\begin{equation}\label{eqn:blow-up-why}
\limsup_{t\to 0^-} u(t,0) = \limsup_{t\to 0^-} \frac{1}{(-t)^\alpha} \tu(0) =\infty \,.
\end{equation} 

Next, we claim that $b \in L^\infty((-1,0); L^1(\R))$ and that 
\begin{equation}
\label{eqn:blow-up-autosimile}
\lim_{t\to 0^-} \|u(t,\cdot) \|_{L^2(\R)} = \infty\,.
\end{equation}
Indeed, for every $t\in [-1,0)$ by the change of variable formula we find
$$ 
\begin{aligned}
\int_{\R} |b(t,x)| \, dx 
&= 
\int_{\R} \Big| \frac{1}{(-t)^{1/2}} \tb\Big(\frac{x}{(-t)^{1/2}}\Big)\Big| \, dx= \int_{\R} |\tb(y)| \, dy \,, \\
\int_{\R} |u(t,x)|^2 \, dx 
&= 
\int_{\R} \Big| \frac{1}{(-t)^{\alpha}} \tu\Big(\frac{x}{(-t)^{1/2}}\Big)\Big|^2 \, dx=\frac{1}{(-t)^{2\alpha-1/2}}  \int_{\R} |\tu(y)|^2 \, dy \,.
\end{aligned}
$$
Since $\alpha<1/2$ and $\tu \in L^2(\R)$, we obtain~\eqref{eqn:blow-up-autosimile} and $u\in L^2((-1,0); L^2(\R))$.
\end{proof} 

\begin{rem}
We point out that the solution $u$ constructed in the proof of Theorem~\ref{thm:counterex-self-similar} 
belongs also to
$u \in L^1((0,1); L^\infty(\R))$, because
$$\|u(t,\cdot)\|_{L^\infty(\R)} = \|u(0,\cdot)\|_{L^\infty(\R)} \cdot \frac{1}{(1-t)^\alpha}$$
and $\alpha \in (0,1)$. This integrability allows for instance to give a meaning to the distributional formulation of \eqref{eqn:our} up to $t=1$ (included).
\end{rem}

\subsection{The case of higher dimension.}

We now extend the result of the previous subsection to higher dimension. We do this by considering solutions with spherical symmetry in the self-similarity variable. 

\begin{thm}\label{thm:counterex-self-similar-d}
Let $d > 1$. There exists a velocity field $b \in L^\infty((0,1); L^d(\R^d)) \cap C^\infty([0,1)\times \R^d)$  
and a smooth solution $u \in C^\infty([0,1)\times \R^d) \cap L^2((0,1); L^{\frac{d}{d-1}}(\R^d))$ 
of equation~\eqref{e:PDE} with $u_0 \in (L^2 \cap L^{\frac{d}{d-1}} \cap L^\infty \cap C^\infty)(\R^d)$ such that 
\begin{equation}\label{e:thesis2}
\begin{aligned}
\limsup_{t\to 1^-} u(t,0) = \infty\,,
\quad &
\lim_{t\to 1^-}  \|u(t,\cdot) \|_{L^{\frac{d}{d-1}}(\R^d)} = \infty \,, \\
&\quad \text{ and } \quad
\lim_{t\to 1^-} \|u(t,\cdot) \|_{L^2(\R^d)} = \infty \,. 
\end{aligned}
\end{equation} 
\end{thm}

\begin{rem}
Observe that the conditions $b \in L^\infty((0,1); L^d(\R^d))$ and $u \in L^2((0,1); L^{\frac{d}{d-1}}(\R^d))$ allow to give a meaning to the distributional formulation of~\eqref{e:PDE}.
\end{rem}

\begin{proof}[Proof of Theorem~\ref{thm:counterex-self-similar-d}]
\noindent {\bf Step 1: Self-similar solutions with spherical symmetry.} 
We look again for a solution $(\tilde u, \tilde b)$ of~\eqref{e:PDE} of the form~\eqref{eqn:self-similar}. The couple $(u,b)$ solves \eqref{e:PDE} if and only if the couple $(\tu, \tb)$ solves for any $y\in \R^d$
$$
\alpha \tilde u(y) + \frac 1 2 y \cdot \nabla \tilde u(y)+ \div (\tilde b \tilde u)(y)= \Delta \tilde u(y) \,,
$$
that is 
\begin{equation}\label{e:previous}
\Big(\alpha- \frac d 2 \Big)\tilde u(y) + \frac 1 2 \div( y \tilde u(y))+ \div (\tilde b \tilde u)(y)= \Delta \tilde u(y)\,.
\end{equation}
Let us consider any function $\tilde U: \R^d \to \R^d$ such that
$$
\div \tilde U  = \tilde u \qquad \text{ in $\R^d$.}
$$
Then~\eqref{e:previous} is satisfied if
$$
\Big(\alpha- \frac d 2 \Big) \tuu(y) + \frac 1 2 y \tu(y)+ \tb(y)\tu(y)= \nabla \tu(y) \,.
$$
We look for solutions such that all the functions are spherically symmetric, namely such that they can be written as
$$
\tilde u (y) = \overline u(|y|)\,, \qquad 
\tilde b(y) = \overline b(|y|) \frac{y}{|y|}\,, \qquad 
\text{ and } \qquad
\tilde U (y) = \overline U(|y|)\frac{y}{|y|} \,,
$$
for $\overline u, \overline v , \overline U \in C^\infty([0,\infty))$.
Therefore, we look for a positive function $\overline u$, a function $\overline U$ solving
\begin{equation}
\label{eqn:bu}
\bu(r) = \overline U'(r)+ \frac{d-1}{r} \overline U(r) \qquad \mbox{in }[0,\infty)
\end{equation}
and a function $\overline b$ given by
\begin{equation}
\label{eqn:bb}
\bb(r) = \frac{\bu'(r)- \Big(\alpha- \frac d 2 \Big) \buu(r) - \frac 1 2 r\bu(r)}{\bu(r)}\,,
\end{equation}
with suitable decay properties at infinity and with a suitable behaviour at $0$ (in order for the corresponding $\tu$ and $\tb$ to be smooth at the origin).

\medskip

\noindent {\bf Step 2: Construction of $\bu$ and $\bb$.}
Given $M>0$ to be chosen later, we look for a solution of the form
$$
\buu (r) = r^{-\gamma} + Cr^{-\gamma-2} \qquad \mbox{for } r >M\,,
$$
for some $\gamma \in (0,1)$ and $C \geq 0$ to be chosen later.
With this choice, for $y>M$ we have
$$
\bu (y) =\overline U'(r)+ \frac{d-1}{r} \overline U(r) =  (-\gamma +d-1) r^{-\gamma-1} + C(-\gamma +d-3) r^{-\gamma-3} \,.
$$
The function $\bu$ is positive for $r$ large enough as soon as $\gamma<d-1$. Correspondingly, we find
$$
\begin{aligned}
& \bb( r) = \frac{1}{-\gamma+d-1+ C(-\gamma+d-3) r^{-2}}  \times \Bigg[ \left(-\alpha+\frac {1+\gamma} 2 \right) r \\
&\qquad\qquad\qquad\qquad\qquad + \left( (\gamma-d+1) (\gamma+1) - C\alpha+\frac {3C} 2 + \frac C 2 \gamma \right)r^{-1}  \\
&\qquad\qquad\qquad\qquad\qquad + C (\gamma-d+3) (\gamma+3) r^{-3} \Bigg] \,.
\end{aligned}
$$
We choose $\alpha$, $C$ in terms of $\gamma$ in order for the function $\bb$ to decay as $r^{-3}$: namely, we require
$$ 
(\gamma-d+1) (\gamma+1) - C\alpha+\frac {3C} 2 + \frac C 2 \gamma= 0
\qquad \text{ and } \qquad
-\alpha+\frac {1+\gamma} 2= 0\,,
$$
that is,
$$
\alpha = \frac{1+\gamma}2
\qquad \text{ and } \qquad
\qquad C = {(-\gamma+d-1)(1+\gamma)}\,.
$$
Finally, we choose $\gamma= d-5/4$, namely
$$
\gamma =d-\frac{5}{4}\,, 
\qquad \alpha= \frac{d}{2} -\frac{1}{8}\,, 
\qquad \text{ and } \qquad
C= \frac{1}{4} \left(d-\frac{1}{4} \right)\,.
$$ 
We also choose
$$
M=\sqrt{2d}+1\,,
$$
so that the function 
$$ 
\frac{1}{4} r^{-d+1/4} - \frac {7}{16} \left( d-\frac{1}{4}\right)r^{-d-7/4}
$$
(which coincides with $\bu(r)$ for $r>M$) is strictly positive in $[M-1, \infty)$.

With this choice we have for $r>M$
\begin{equation}\label{e:outside}
\begin{aligned}
\buu (r) &= r^{-d+5/4} + \frac{1}{4} \left( d-\frac{1}{4} \right)r^{-d-3/4}\,, \\
\bu (r) & = \frac{1}{4} r^{-d+1/4} - \frac{7}{16} \left(d-\frac{1}{4} \right)r^{-d-7/4}\,, \\
\bb(r) & = \frac{7(4d-1)(4d+7) r^{-3}}{4 \big( 16-7(4d-1)r^{-2} \big)} \,.
\end{aligned}
\end{equation}

Finally, we fix a constant $L$ depending on the dimension $d$ only such that
\begin{equation}
\label{eqn:L-condition}
Lr \leq  r^{-d+5/4} + \frac {1}{4} \left( d-\frac{1}{4} \right)r^{-d-3/4} \qquad \text{ for $r\in [M-1,M]$}
\end{equation}
(notice that the right-hand side coincides with the expression of $\buu$ in \eqref{e:outside}).
We consider a smooth function $\psi(r)$ which is a cutoff between $M-1$ and $M$, namely it coincides with $1$ in  $[0,{M-1}]$ and with $0$ in $[M, \infty)$. We define 
$$
\buu(r) = \psi(r) Lr + (1-\psi(r))\left[r^{-d+5/4} + \frac{1}{4} \left( d - \frac{1}{4} \right)r^{-d-3/4}\right] 
\qquad
\text{ for $r\geq 0$.}
$$
In this way, $\buu$ coincides with $Lr$ in a neighbourhood of $0$, and consequently $\tuu(x) = \buu(|x|)\frac{x}{|x|}$ is smooth. Moreover the associated $\bu$, defined as in \eqref{eqn:bu}, is positive in $[0,\infty)$ (and locally constant in a neighbourhood of the origin) because of \eqref{eqn:L-condition} and of the positivity of 
$$
\div(Lr) 
\qquad \text{ and } \qquad
\div \left[r^{-d+5/4} + \frac{1}{4} \left( d - \frac{1}{4} \right)r^{-d-3/4}\right] \,.
$$
Indeed, $\bu$ is given by
\begin{equation*}
\begin{aligned}
\bu(r) &= 
\overline U'(r)+ \frac{d-1}{r} \overline U(r) \\
& =
\psi(r) Ld+ (1-\psi(r))\left[\frac{1}{4} r^{-d+1/4} - \frac{7}{16} \left( d-\frac{1}{4}\right)r^{-d-7/4}\right] \\
& \qquad -\psi'(r) \left[r^{-d+5/4} + \frac{1}{4} \left( d-\frac{1}{4} \right)r^{-d-3/4} - Lr \right]
\end{aligned}
\end{equation*}
and all three terms in the right-hand side are positive and the function $\psi$ is decreasing. Correspondingly we also define $\bar b$ according to \eqref{eqn:bb} close to the origin and notice that it is smooth.

Finally, as regards the integrability of $\bu$ and $\bb$, thanks to their decay properties at infinity \eqref{e:outside}, we remark that
\begin{equation}
\label{eqn:integr-tb-tu}
\int_0^\infty \left( \bu^{\frac{d}{d-1}} + \bu^2\right) r^{d-1} \, dr <\infty 
\qquad \text{ and } \qquad
\int_0^\infty |\bb|^d r^{d-1}\, dr <\infty \,.
\end{equation}

\medskip
\noindent {\bf Step 3: Bounds on $b$ and finite-time blow-up of $u$.} 
From the properties of $\bu$ and $\bb$ in Step~2 the functions $\tu$ and $\tb$ are smooth as well, while the functions $b$ and $u$ defined in~\eqref{eqn:self-similar} are in $ C^\infty([-1,0)\times \R^d)$ (in particular, also $b$ is smooth at the origin since $\bb(r)$ has a linear behaviour with respect to $r$ around $0$). Moreover, the function $u$ blows up for $t=0$ as it happened in~\eqref{eqn:blow-up-why}.

Regarding the integrability of $b$, we have that $b \in L^\infty((-1,0); L^d(\R^d))$ because for every $t\in [-1,0)$ by the change of variable formula
$$ 
\begin{aligned}
\int_{\R^d} |b(t,x)|^d \, dx &= \int_{\R^d} \frac{1}{(-t)^{d/2}}  \left| \tb \left(\frac{x}{(-t)^{1/2}} \right)\right|^d \, dx \\
& = \int_{\R^d} |\tb(y)|^d \, dy  = C_d \int_0^\infty |\bb(r)|^d r^{d-1}\, dr \,,
\end{aligned}
$$
which is finite thanks to the second property in \eqref{eqn:integr-tb-tu}.
Finally, we observe that $u_0 \in (L^2 \cap L^{\frac{d}{d-1}})(\R^d)$ by~\eqref{eqn:integr-tb-tu} and we notice that
$$
\begin{aligned}
\int_{\R^d} |u(t,x)|^{\frac{d}{d-1}} \, dx 
&= \int_{\R^d} \left| \frac{1}{(-t)^{\alpha}} \tu\left(\frac{x}{(-t)^{1/2}}\right)\right|^{\frac{d}{d-1}} \, dx \\
&= \frac{1}{(-t)^{\alpha\frac{d}{d-1}-d/2}}  \int_{\R} |\tu(y)|^{\frac{d}{d-1}} \, dy \\
& =\frac{C_d}{(-t)^{\alpha\frac{d}{d-1}-d/2}}  \int_0^\infty |\bu(r)|^{\frac{d}{d-1}} r^{d-1} \, dr \,.
\end{aligned}
$$
Since $\alpha>(d-1)/2$ the previous quantity blows up, giving the second statement in~\eqref{e:thesis2}, while we have $u\in L^2((-1,0); L^{\frac{d}{d-1}}(\R))$ thanks to the computation
$$
\int_{-1}^0 \left(\int_{\R^d} |u(t,x)|^{\frac{d}{d-1}} \, dx \right)^{ \frac{2(d-1)}d}\, dt 
= 
\int_{-1}^0 \frac{C_d}{(-t)^{2\alpha-d+1}}\,dt \left( \int_0^\infty |\bu(r)|^{\frac{d}{d-1}} r^{d-1} \, dr \right)^{ \frac{2(d-1)}d} \,.
$$
The same type of computation allows to prove the last statement in~\eqref{e:thesis2}. 
\end{proof}


\subsection*{Acknowledgment}
The authors thank Paolo Baroni, Ugo Gianazza, Renato Luc\`a, and Connor Mooney for several interesting discussions. GC was partially supported by the Swiss National Science Foundation (Grant 156112) and by the European Research Council (ERC Starting Grant 676675 FLIRT). LVS is a member of the GNAMPA group of INdAM (``Istituto Nazionale di Alta Matematica").

\end{document}